\documentclass{article}
\usepackage{amssymb}
\usepackage{amsmath}
\usepackage{amscd}
\usepackage{amsthm}
\usepackage{graphicx}

\usepackage{color}          
\usepackage{epsfig}
\usepackage{authblk}

\newcommand{\C} {\mathbb C}

\newcommand{\mc}{\mathcal}

\newtheorem{theorem}{Theorem}
\newtheorem{lemma}[theorem]{Lemma}
\newtheorem{corollary}[theorem]{Corollary}
\newtheorem{proposition}[theorem]{Proposition}
\newtheorem*{definition}{Definition}
\newtheorem*{Example}{Example}

\begin{document}
\title{Semigroup representations in holomorphic dynamics.}
\author[1]{Carlos Cabrera}
\author[1]{Peter Makienko}
\author[2]{Peter Plaumann}

\affil[1]{Instituto de Matem\'aticas\\ Unidad Cuernavaca. UNAM}

\affil[2]{Mathematisches Institut\\ Friedrich-Alexander-Universit\"{a}t Erlangen-N\"{u}rnberg}

\maketitle

\footnotetext{This work was partially supported by PAPIIT project IN 100409.}



\begin{abstract}
We use semigroup theory to describe the group of automorphisms of some
semigroups of interest in holomorphic dynamical systems. We show, with some
examples, that representation theory of semigroups is related to usual constructions in
holomorphic dynamics. The main tool for our discussion is a theorem due to
Schreier. We extend this theorem, and our results in semigroups, to the setting 
of correspondences and holomorphic correspondences. 
\end{abstract}

\section{Introduction}
One of the motivations of this paper is to add a new entry to Sullivan's dictionary
between holomorphic dynamics and Kleinian groups. This entry consist of  the algebraic part
of a holomorphic dynamical system, that corresponds to the algebraic structure of
a Kleinian group. From this point of view, the natural object is a semigroup.
This investigation was initiated by the following question of \'{E}tienne Ghys:
\begin{itemize}
 \item Are there multiplicative characters, defined on the semigroup of
polynomials with composition, which do not arises as a multiplicative function
of the degree?
\end{itemize}
In general, multiplicative characters play an important role in  representation
theory, which is also the subject of this work.

We give a positive answer to Ghys' question, and suggest a general method 
to construct multiplicative characters on polynomials. Then, we study the automorphism
groups of several semigroups of interest in holomorphic dynamical systems. 
In the second section, we show that the semigroup of polynomials and rational 
maps  are generated by linear automorphisms and the Galois group action on $\C$.
A result, due to Hinkkanen, states that the automorphism group of entire
functions consist of the  group of continuous  inner automorphisms. 
Using Hinkkanen's Theorem,  we show that, the group of automorphisms 
of the semigroup of meromorphic functions also consists of the group of
continuous inner automorphisms. We give algebraic conditions, using sandwich
semigroups, that characterize when two given polynomials, or rational maps, are
conformally conjugated. 
The main tool  is a theorem, due to Schreier, which states that any
representation of a semigroup of maps $S$ is geometric, whenever $S$ contains
constant maps. This result is useful to study the representation space of
semigroups of maps and correspondences. Hence, this theorem remarks the
importance of considering semigroups of maps together with constants. This point
of view is also adopted in Eremenko's paper \cite{EreEnd}.

On the third section we provide several examples of semigroups representations
that appear in holomorphic dynamical systems.  Among these, there is a
connection with the deformation space of a given rational map. The
topology of the deformation space of a rational space is discussed in
\cite{Cabre.Maki.Teic}. 

On the last section, we generalize the results of Section 2, to the setting of
holomorphic correspondences. In particular, we prove a generalized Schreier
Lemma for correspondences and holomorphic correspondences. This allow us to 
characterize the Galois group $Gal(\C)$, the group of all field automorphisms of
$\C$, as the subgroup of $Bij(\C)$ that under conjugation functionally
preserves the finite holomorphic world.  

\section{Semigroups.}

A semigroup is a set $S$ together with a binary operation
which is associative. Given any set $X$, consider the semigroup  
$Map(X)$ of all maps $\phi:X\rightarrow X$, with composition as semigroup
operation. It contains the group of bijections $Bij(X)$. A subset $I\subset S$
is
called a left (respectively right) ideal, if $s i\in I$ (respectively $i s\in
I$), for all $s\in S$ and $i\in I$.

For every element $s$ in $S$, let $\tau_s$ be the left translation by $s$.
The map $s\mapsto \tau_s$ induces a representation $\phi$ from $S$ into
$Map(S)$. In fact, the same map gives a representation from $S$ into $Map(I)$ for
every left ideal $I$ in $S$. Note that $\phi$ is a faithful representation only
in the case when, for every pair of elements $g, h$ in $S$, there is $i\in I$
such that $gi\neq hi$. 

Let $X$ be an abstract set, then there is a canonical inclusion of $X$ into
$Map(X)$, sending  every $x$ in $X$ to the constant map $x$.
The image of this map is a left ideal  $\mc{I}$ in $Map(X)$, the \textit{ideal
of constants} of $X$. An element $x$ in $Map(X)$ belongs to $\mc {I}$ if, and
only if, for every $g\in Map(X)$ we have :

\begin{itemize}
 \item[i)] $g\circ x\in \mc{I}$;
 \item[ii)] $x\circ g=x$.
\end{itemize}

The ideal of constants is contained in every left ideal of $Map(X)$. In this sense,
the ideal of constants is the smallest left ideal in $Map(X)$. Given any
semigroup $S$, we can use properties (i) and (ii) to define the ideal
of constants, whenever it exists.

\begin{Example}
In general, the ideal of constants in $Map(X)$ is not prime. Assume that $X$ has at
least three points $x_1,x_2$ and $x_3$, let  $g_1$ and $g_2$ in
$Map(X)$, such that, $Image(g_1)=\{x_1,x_2\}$, $g_2(x_1)=g_2(x_2)\neq g_2(x_3)$,
then $g_2\circ g_1$ is constant but $g_2$ is not. However, if we restrict to the
space of continuous maps $C_0(X)$ in a topological space $X$ with enough
regularity, then the ideal of constants is prime. 
\end{Example}

From now on, we will consider the special case where $X$ is either the complex
plane $\C$ or the Riemann sphere $\bar{\C}$. Given two subsets $S_1$ and
$S_2$ in $Map(X)$, we denote by $\langle S_1, S_2 \rangle$ the semigroup
generated by $S_1$ and $S_2$.

\subsection{Multiplicative characters of semigroups.}

Let $Pol(\C)$ denote the semigroup  of complex polynomials 
with composition as semigroup multiplication. Let us consider the set
 $Hom(Pol(\C),\C)$ of all multiplicative characters, that is, the set of
homomorphisms $\chi$ in  satisfying $$\chi(P_1\circ
P_2)=\chi(P_1)\cdot\chi(P_2)$$  for all $P_1,P_2\in Pol(\C)$. 

The degree function $deg$, is a basic example of a multiplicative character in
$Pol(\C)$. Any multiplicative function of $deg$ induces a multiplicative
character. It was a question of \'{E}. Ghys whether there are other characters
apart from these examples.  We give a positive answer to this question and give
a description of how to construct multiplicative characters on $Pol(\C)$. To do
so, first let us recall a theorem due to Ritt,  see \cite{Ritt}.

\begin{definition}
A polynomial $P$, is called \textnormal{prime}, or
\textnormal{indecomposable},
if whenever we have $P=Q\circ R$, where $Q$ and $R$ are polynomials,  then
either $deg(Q)=1$ or $deg(R)=1$. A decomposition of $P=P_1\circ P_2 \circ ...
\circ P_n$ is called a \textnormal{prime decomposition} if, and only if, each
$P_i$ is a prime polynomial of degree at least $2$ for all $i$.
\end{definition}

Given a prime decomposition of a polynomial $P=P_1\circ ... \circ P_n$, a
\textit{Ritt transformation}, say in the $j$ place, is the substitution of the
pair $P_j\circ P_{j+1}$, in the prime decomposition of $P$, by the pair
$Q_j\circ Q_{j+1}$. Where $Q_j$ and $Q_{j+1}$ are prime polynomials satisfying
$P_j\circ P_{j+1}=Q_j\circ Q_{j+1}$. Now we can state Ritt's theorem.

\begin{theorem}[Ritt]\label{Ritt.thm}
Let $P=P_1\circ P_2\circ ... \circ P_m$ and $P=Q_1\circ ...\circ Q_n$ be two
prime decompositions of $P$, then $n=m.$ Moreover, any two given prime
decompositions of $P$ are related by a finite number of Ritt transformations.
\end{theorem}

In \cite{Ritt}, Ritt showed that there are  three types of  Ritt
transformations, namely, see also \cite{BeaRitt}:

\begin{enumerate}
 \item  Substitute $P_i\circ P_{i+1}$ by $(P_i\circ A)\circ (A^{-1})\circ
P_{i+1}$, where $A$ is an affine map.
 \item Substitute  $P_i\circ P_{i+1}$ by $P_{i+1}\circ P_{i}$, when $P_i$ and
$P_{i+1}$ are Tchebychev polynomials.
\item If $P_i(z)=z^k$ and $P_{i+1}(z)=z^r P(z^k)$ for some polynomial $P$ and
natural numbers $r$ and $k$. Define $Q_{i+1}(z)=z^r(P(z))^k$, then substitute
$P_i\circ P_{i+1}$ by $Q_{i+1}\circ P_i$. 
\end{enumerate}

In particular, there are two invariants of a prime decomposition, the length of
a prime decomposition, and the set of degrees in the prime decomposition. Hence,
for every $P\in Pol(\C)$  the length of a prime decomposition of $P$ is a
well defined additive character $l(P)$.  That is, it satisfies
$$l(P_1\circ P_2)=l(P_1)+l(P_2).$$ Now, define the function $\chi$ by
$\chi(P)=e^{l(P)}$. Then, $\chi$ is a multiplicative character which is not
a multiplicative function of the degree. 

The following theorem gives a method to generate multiplicative characters in
$Pol(\C)$.

\begin{theorem}\label{th.characters} Let $\phi$ be a complex function, defined
on the set of prime polynomials, satisfying:
\begin{itemize}
 
\item[(i)] $\phi(c)=0$ for every constant $c$.
\item[(ii)] If $P_1,P_2,P_3, P_4$ are prime
polynomials with $P_1\circ P_2=P_3\circ P_4$,
then $$\phi(P_1)\cdot\phi(P_2)=\phi(P_3)\cdot\phi( P_4).$$
\end{itemize}

Then, $\phi$ generates a multiplicative character $\Phi$. Conversely, if
$\Phi$ is a multiplicative character in $Pol(\C)$, which  is not the constant
map $1$, then $\Phi$ satisfies the conditions above.

\end{theorem}

\begin{proof} Let $P$ be a composite polynomial and $P=P_1\circ P_2\circ ...
\circ P_n$ be a prime decomposition of $P$,  define $$\Phi(P)=\phi(P_1)\cdot
\phi(P_2)\cdot ...\cdot \phi(P_n).$$ Let us check that $\Phi$ is well defined.
By Theorem \ref{Ritt.thm},  it is enough to consider a step modification of
$P_1\circ P_2\circ ... \circ P_n$. Let $Q_j$ and $Q_{j+1}$ be two polynomials
such that $P_j\circ P_{j+1}=Q_j\circ Q_{j+1}$, then by condition (ii) we have
$\phi(P_j)\cdot \phi(P_{j+1})=\phi(Q_j)\cdot \phi(Q_{j+1})$, in consequence 
$$\phi(P_1)\cdot \phi(P_2)\cdot ...\cdot \phi(P_j)\cdot \phi(P_{j+1})\cdot
... \cdot\phi(P_n)=$$$$\phi(P_1)\cdot \phi(P_2)\cdot ...\cdot \phi(Q_j)\cdot
\phi(Q_{j+1})\cdot...\cdot
\phi(P_n).$$ Hence $\Phi$ is invariant under step modifications and, 
by Theorem \ref{Ritt.thm}, it is independent of the prime decomposition 
of $P$. It follows, from the definition, that $\Phi$ 
is a multiplicative character.

Conversely, let $\Phi$ be a multiplicative character. For any pair of 
constants $c_1$ and $c_2$, the equations  $c_1\circ c_2=c_1$ and
 $c_2\circ c_1=c_2$ imply 
$$\Phi(c_1)=\Phi(c_1\circ c_2)=\Phi(c_1)\cdot \Phi(c_2)
$$$$=\Phi(c_2)\cdot \Phi(c_1)=\Phi(c_2\circ c_1)=\Phi(c_2).$$ Then, for
every 
constant $c$, either we have $\Phi(c)=1$ or $\Phi(c)=0$. If $\Phi(c)=1$, the
equation $P(c)=P\circ c$ implies that $\Phi(P)=1$ for all $P$. Hence if $\Phi$
is not constantly $1$, then we have $\Phi(c)=0$ for every constant $c$. The
second condition follows from the fact that $\Phi$ is a multiplicative
character.
 \end{proof}

\begin{Example}[Affine characters.]
Let $H$ be the ideal of non injective polynomials. Any multiplicative
character $\chi :\textnormal{Aff}(\C)\rightarrow \C$ admits an
extension to a multiplicative character defined in $Pol(\C)$. For instance, put
$\chi(c)=0$ for all constant $c$, and $\chi(h)=0$, for
all other $h$ in $H$. 
\end{Example}

In the same way, we can extend affine characters to other semigroups containing 
$\textnormal{Aff}(\C)$, such as $Rat(\C)$, $Ent(\C)$ or the semigroup of
holomorphic correspondences discussed at the end of this work. 

Now let us construct non-trivial extension of the constant affine character
equal to $1$. In order to do so, we have to consider the bi-action, left and
right, of $\textnormal{Aff}(\C)$ on $Pol(\C)$. The bi-orbit of a polynomial $P$
is the set of
all polynomials of the form $A\circ P \circ B$, where $A, B$ belong to
$\textnormal{Aff}(\C)$. We say that a polynomial has no symmetries if, there are
no elements $A,B$, in $\textnormal{Aff}(\C)$, such that $P=A\circ P \circ B$. 

\begin{lemma}\label{lemma.affine}
 Let $P$ be a prime polynomial, and let $\mathcal{AF}(P)$ be the semigroup
generated by the bi-orbit of the Affine group of the set of iterates $\{P^n\}$.
Let $Q$ and $R$ be a pair of polynomials, of degree at least $2$, such that
$Q\circ R\in \mathcal{AF}(P)$, then $Q\in\mathcal{AF}(P) $ and
$R\in\mathcal{AF}(P)$.
\end{lemma}

\begin{proof}
Since $Q\circ R$ belong to $\mathcal{AF}(P)$, there is a prime decomposition of
$Q\circ R$ whose elements are of the form $A\circ P \circ B$. By
Ritt's Theorem, any other prime decomposition of $Q\circ R$ is obtained by a
finite number of Ritt's transformations. But,  Ritt's transformations are
either permutations, or substitution by a pair of elements in the bi-affine
orbit. Hence, all prime decompositions of $Q\circ R$ have prime elements in
$\mc{AF}(P)$. Then the conclusion of the Lemma follows. 
\end{proof}

\begin{Example}
 Let $\chi$ be the constant multiplicative character equal to $1$ 
defined on $\textnormal{Aff}(\C)$, and $P$ be a prime polynomial of degree at
least $2$. Let us extend $\chi$ to all $Pol(\C)$ defining $\chi(Q)=1$ for all
$Q$ in the bi orbit by $\textnormal{Aff}(\C)$ of the set $\{P^n\}_{n\in
\mathbb{P}}$, and
$\chi(Q)=0$ for all the other polynomials $Q$ in $Pol(\C)$.  By Lemma
\ref{lemma.affine} and Theorem \ref{th.characters}, this is a well defined
character. In fact for any number $a$, defining $\chi(A\circ P^n
\circ B)=a^n$ where $A, B$ are elements in $\textnormal{Aff}(\C)$ gives other
extensions of $\chi$ in $Pol(\C)$.

To extend arbitrary multiplicative characters defined on $\textnormal{Aff}(\C)$,
the construction of the character is more involved. At least in the case where
$P$ is a prime polynomial, such that every iterate $P^n$ is without symmetries,
it is possible to extend any multiplicative character $\chi$ on
$\textnormal{Aff}(\C)$. 
\end{Example}

The ideal of constants is very useful to understand the structure of
$Map(X)$. A homomorphism $\phi:Map(X)\rightarrow Map(Y)$ is called
\textit{geometric} if, there is a map $f:X\rightarrow Y$ satisfying
$\phi(P)\circ f =f\circ P$ for every $P\in Map(X)$. Now we recall a result due  
to Schreier that describes the semigroup $Map(X)$  using the ideal of
constants. 
For further details see  \cite{Schre.end}, and also  the discussions in Eremenko's 
paper \cite{EreEnd} and Magill's survey \cite{Mag1}.

\begin{lemma}[Schreier's Lemma]\label{StWier.lemma} 
Let $\phi:Map(X)\rightarrow Map(Y)$ be a homomorphism, then $\phi$ is
geometric. In the case where $\phi\in Aut(Map(X))$ and $\phi(P)\circ f= f\circ
P$,  then $f$ is a bijection of $X$ and $\phi(P)=f\circ P \circ f^{-1}$, for all
$P\in Map(X)$. 
\end{lemma}

\begin{proof}
Consider the restriction $f:=\phi_{|X}$ to the ideal of constants. Since $\phi$ is
a homomorphism, it maps ideals into ideals, it also preserves the properties of
the ideal of constants, hence $f$ sends constants to constants. So 
$f$ is a map from $X$ to $Y$. Moreover, $$\phi(P(x))=f(P(x))$$ since $P(x)\in X$. Also,
$$\phi(P(x))=\phi(P\circ x)=\phi(P)\circ f(x)=\phi(P)(f(x)),$$ hence
$$\phi(P)(f(x))=f(P(x)).$$ If $\phi\in Aut(Map(X))$, then $f$ is a map from
$X$ to itself. Moreover, since $\phi$ is an automorphism, we can apply the
argument to $\phi^{-1}$, so we get that $f$ is invertible. Which implies that
$f$ is a bijection and the formula $\phi(P)=f\circ P \circ f^{-1}$.
\end{proof}

In fact, there is no need that the homomorphism in Lemma \ref{StWier.lemma} is 
defined in all $Map(X)$, the same proof above shows.

\begin{corollary}\label{Sch.coro} Let $S_1$ and $S_2$ be subsemigroups of 
$Map(X)$ and $Map(Y)$, respectively, and such that $X_1=S_1\cap X$ and
$Y_1=S_2\cap Y$ are both non empty sets. If $\phi:S_1\rightarrow S_2$ is a
homomorphism, then there exist $f:X_1\rightarrow Y_1$, such that for all
$h\in S_1$, $\phi(h)\circ f=f\circ h$. Moreover, 
\begin{itemize}
 \item the homomorphism $\phi$ is injective, or surjective, if and only
if, the map $f$ is injective or surjective. In particular, $\phi$ is an
isomorphism if, and only if, $f$ is a bijection.

\item When $S_1$ and $S_2$ are
topological semigroups, then $\Phi$ is continuous if, and only if, $f$ is
continuous.

\end{itemize}
 \end{corollary}

Along with the ideal of constants, the affine group $\textnormal{Aff}(\C)$ plays
an important role in the description of automorphisms of polynomials. Later on,
we will consider generalizations to semigroups generated by correspondences. A
particular case of Lemma \ref{StWier.lemma} is the following

\begin{corollary}
For any set $X$, the group $Aut(Map(X))$ is isomorphic to $Bij(X)$.
\end{corollary}
 
Let $Gal(\C)$ denote the absolute Galois group of $\C$, that is, the full group of field
automorphisms of $\C$. Remind that since every orientation preserving element in
$Gal(\C)$ must fix the complex rationals, the identity and complex conjugation
are the only continuous elements in $Gal(\C)$.
The action of $Gal(\C)$ extends to an action in $Rat(\C)$, the semigroup of
rational functions in $\C$.  In particular, the action of $Gal(\C)$  in $\C$
extends to an action in $Pol(\C)$.

\begin{proposition}\label{prop.aut.pol} The group of automorphisms of
$Pol(\C) $ is generated by $Gal(\C)$ and $\textnormal{{Aff}}(\C)$. Moreover,
$Aut(Pol(\C))=Aut(\textnormal{Aff}(\C))$.
\end{proposition}
\begin{proof}

Let $\phi$ be an element of $Aut(Pol(\C))$. By 
Lemma \ref{StWier.lemma}, the
restriction $f=\phi_{|_\C}$ is a bijection from $\C$ to $\C$, and
$\phi(P)=f\circ P \circ f^{-1}.$ First, let us check that  $\phi=Id$ if, and only if,
$f=Id$. Note that we can realize evaluation as composition with a constant function.
If $f=Id$,  then we have $$\phi(P)(z)=\phi(P)\circ
f(z)=\phi(P(z))=f(P(z))=P(z)$$ for every polynomial $P$ and $z\in \C,$ that is
$\phi=Id$. The converse is clear.

Since, by Lemma \ref{StWier.lemma}, $\phi$ is a conjugation, then $\phi(P)$ and
$P$ have the same degree. In particular, $\phi$ leaves the affine
group $\textnormal{Aff}(\C)$ invariant, so
$\phi(\textnormal{Aff}(\C))=\textnormal{Aff}(\C)$.  This fact also
follows from the characterization of $\textnormal{Aff}(\C)$ as the set of
injective polynomials. In particular, $Aut(Pol(\C))\subset
Aut(\textnormal{Aff}(\C))$, the converse is also true by Lemma
\ref{StWier.lemma}, since any conjugacy in the $\textnormal{Aff}(\C)$ extends to
a conjugacy in $Pol(\C)$.

The group of translations $T$ is the commutator of $\textnormal{Aff}(\C)$, hence
$T$ is invariant under $\phi.$ The value of a translation at one point,
determines the translation. Let $\tau_c$ denote the translation $z\mapsto z+c$,
since $$\phi(\tau_c)(f(0))=f\circ \tau_c\circ
f^{-1}(f(0))=f(c),$$ then $$\phi(\tau_c)=\tau_{(f(c)-f(0))}.$$

Define $g(z)=f(z)-f(0)$, then $g$ is a bijection of $\C$ which is the restriction to the
constants of the map
$\tilde{\phi}=\tau_{-f(0)}\circ \phi$ and $g(0)=0$. By definition, 
$\tilde{\phi}\in
Aut(Pol(\C))$ and $\phi(\tau_c)=\tau_{g(c)}$, it follows that
$$\phi(\tau_{c_1+c_2})=\phi(\tau_{c_1})\circ \phi(\tau_{c_2}),$$ that is,
$$g(c_1+c_2)=g(c_1)+g(c_2).$$ 

Let $A_0$ be the group of injective polynomials
fixing $0$, since $\tilde{\phi}(0)=0$, then 
$\tilde{\phi}(A_0)=A_0$.
Now we repeat the argument above, this time in multiplicative terms, to show
that $h(c)=\frac{g(c)}{g(1)}$ is a bijection of $\C$ preserving multiplication
and $h(1)=1$. By definition, $h$ also preserves addition with $h(0)=0$, hence
$h\in Gal(\C)$. Note that $h$ is the restriction to constants of the map
$(g(1)^{-1} \tau_{-f(0)})\circ \phi$. This implies that $f=g(1) h +f(0)$ as we
wanted to show.

 \end{proof}

The proof of Proposition \ref{prop.aut.pol} can be adapted to show 

\begin{proposition}\label{prop.aut.rat} Let $Rat(\bar{\C})$ denote the semigroup of
rational maps in the Riemann sphere, then $Aut(Rat(\bar{\C)})=\langle
Gal(\C),PSL(2,\C)\rangle$
\end{proposition}

\begin{proof}
Since $\phi(Id)=Id$, and using the formula $R\circ R^{-1}=Id$, one can check
that $\phi$ sends $PSL(2,\C)$, the group of invertible rational maps, into
$PSL(2,\C)$. Post composing $\phi$ with an element of $PSL(2,\C)$ we can assume
that $\phi(\infty)=\infty$. In this case, it follows
that $\phi(\textnormal{Aff}(\C))\subset \textnormal{Aff}(\C)$, hence if
$\phi(\infty)=\infty$ then $\phi\in Aut(\textnormal{Aff}(\C))$. Since every
element in  $\langle Gal(\C),PSL(2,\C)\rangle$ induces a conjugation in
$Rat(\C)$, we have the claim of the proposition.

\end{proof}
 
Now we want to study the semigroup of meromorphic functions $Mer(\C)$.
This semigroup contains the  semigroup of entire functions $Ent(\C)$. We recall
a theorem by Hinkkanen \cite{Hinkentire}.

\begin{theorem}[Hinkkanen]\label{Hink} Let $\phi$ be a geometric automorphism of
$Ent(\C)$, then $\phi$ is affine. 

\end{theorem}
In other words, except for the identity, no element in $Gal(\C)$ leaves the
semigroup $Ent(\C)$ invariant in the space of formal series. The following are
immediate consequences of Lemma \ref{StWier.lemma} and Hinkkanen's Theorem.

\begin{proposition} The group of automorphisms of $Mer(\C)$ is isomorphic
to $PSL(2,\C)$. 
 \end{proposition}

\begin{proof} Let $\phi$ be an element in $Aut(Mer(\C))$, and $\gamma$ be
an
element in $PSL(2,\C)$ so that $\gamma(\phi(\infty))=\infty$. 
By Lemma \ref{StWier.lemma}, $\gamma\circ \phi$ is a geometric automorphism 
in $Mer(\C)$. Now, a meromorphic map $g$ is entire if, and only if, $g$ has  no finite poles. Since infinity is  fixed by $\gamma\circ \phi$, the map  $\gamma\circ \phi$
sends entire functions into entire functions.  By Theorem \ref{Hink}, we have
$\gamma \circ \phi\in \textnormal{Aff}(\C)$ and  $\phi\in PSL(2,\C)$.  
 \end{proof}

\begin{corollary}\label{cor.mer.cont}
Every automorphism of $Mer(\C)$ is continuous. 
\end{corollary}

\begin{corollary}\label{cor.ext.cont} A map $\phi$ in $Aut(Rat(\C))$ is
continuous if, and only if, $\phi$ extends to a map in $Aut(Mer(\C))$.
\end{corollary}

All above gives  a characterization of elements in $Bij(\C)$ that belong to the
Galois group $Gal(\C)$.

\begin{theorem}\label{Gal.thm} Let $F$ be an element in  $Bij(\C)$, let us
assume it 
fixes three points in $\C$, then the following are equivalent.
 \begin{itemize}
  \item[i)]  The map $F$ belongs to $Gal(\C)$.
  \item[ii)] The induced map in $Map(\C)$ sends $Rat(\C)$ into itself.
   \item[iii)] The induced map in $Map(\C)$ sends $Pol(\C)$ into itself. 
\item[iv)] The induced map in $Map(\C)$ sends $\textnormal{Aff}(\C)$ into
itself.
 \end{itemize}

\end{theorem}

\subsection{Sandwich semigroups.}
Here, we give an algebraic condition for when two
rational maps are M\"{o}bius conjugated, for this we do not require any
dynamical restrictions on the rational maps. We start with the polynomial case,
where the action of $PSL(2,\C)$ is replaced by the action of
$\textnormal{Aff}(\C)$. 

Given a map $g:Y\rightarrow X$, let us define
on $Map(X,Y)$ the following
operation, for $f,h\in Map(X,Y)$ put $f\ast_g
h= f\circ g \circ h$. We denote this new semigroup by
$Map_g(X,Y)=(Map(X,Y),\ast_g)$. In
particular, if $S$ is a subsemigroup of $Map(X)$ and $g\in Map(X)$, 
the set $S_g:=(S,\ast_g)$ is also a semigroup. In particular, given a polynomial
$P$,  let us consider the semigroup $Pol_P(\C).$

\begin{theorem}\label{Thm.Sandwich} Let $P_1$ and $P_2$ be two complex
polynomials. Let $$\Phi:Pol_{P_1}(\C)\rightarrow Pol_{P_2}(\C)$$ be an
isomorphism of semigroups. Then there is $f\in Bij(\C)$, and $B\in
\textnormal{Aff}(\C)$, such that $\Phi(P)=f\circ P \circ  f^{-1}\circ B^{-1}$. 
 
\end{theorem}
\begin{proof} We first check that
$\phi(\textnormal{Aff}(\C))=\textnormal{Aff}(\C)$. By definition, for every pair
of polynomials $P,Q$, we have $$\phi(P\ast_{P_1}  Q)=\phi(P)\ast _{P_2}\phi(Q).$$ 
Let $f=\phi|_\C$ then, taking for $Q$ a constant $c\in \C$, the equality
above becomes

\begin{equation}\label{pri.eq}
  f(P\circ P_1(c))=\phi(P)\circ P_2(f(c)),
 \end{equation}
for every polynomial $P\in Pol(\C)$. Since $\phi$ is an isomorphism, $f$ is an invertible map. Hence the equation above implies that $f$ conjugates the polynomial $P\circ P_1$ to $\phi(P)\circ P_2$. Then $deg(P\circ
P_1)= deg(\phi(P)\circ P_2)$. We obtain a similar equation for $\phi^{-1}$
$$f^{-1}(P\circ P_2(c))=\phi^{-1}(P)\circ P_1(f^{-1}(c))$$ and $deg(P\circ
P_2)= deg(\phi(P)\circ P_1)$. Since $deg$ is a multiplicative character, and takes values in
$\mathbb{N}$, for every invertible polynomial $A$ we obtain
$$deg(P_1)=deg(\phi(A))\cdot deg(P_2)$$ and $$deg(P_2)=deg(\phi^{-1}(A))\cdot
deg(P_1).$$ Hence $1=deg(\phi(A))\cdot deg(\phi^{-1}(A))$, which implies that
$\phi(\textnormal{Aff}(\C))=\textnormal{Aff}(\C)$. 

Define $B=\phi(Id)$, then $B$ is an element of $\textnormal{Aff}(\C)$, now consider the
map $\phi_B:(Pol(\C),
P_2)\rightarrow (Pol(\C), B^{-1}P_2),$ given by $\phi_B(P)=P\circ B$. The $\phi_B$ is an isomorphism of semigroups. Then the composition $\Phi=\phi_B\circ \phi$
is an isomorphism from $(Pol(\C), P_1)$ to $(Pol(\C), B^{-1}P_2)$, satisfying
$\Phi(Id)=Id$. Last equation implies that  $\Phi(P_1)=P_2$. Moreover, since
 $\Phi(c)=\phi(c)\circ B=\phi(c)=f(c)$, the restrictions to constants, of the maps 
$\phi$ and $\Phi$, are equal. If $P=Id$ in
(\ref{pri.eq}), we obtain that $P_1=f^{-1}\circ P_2\circ  f$, which implies from
(\ref{pri.eq}) that for all 
$c\in C$
$$f\circ P\circ P_1(f^{-1}(c))=\Phi(P)\circ P_2(f(f^{-1}(c))$$ $$=\Phi(P)\circ P_2(c),$$ 
then   $\Phi(P)=f\circ P \circ f^{-1}$. Hence $\phi(P)=f\circ P \circ
f^{-1}\circ B^{-1}$ as we wanted to show.

\end{proof}
\begin{corollary} Two polynomials $P_1$ and $P_2$ are affinely conjugate if, and
only if, the semigroups $Pol_{P_1}(\C)$ and
$Pol_{P_2}(\C)$ are continuously isomorphic with an
isomorphism $\phi$, such that  $\phi(Id)=Id$.
\end{corollary}

By substituting $\textnormal{Aff}(\C)$ by $PSL(2,\C)$, and $Pol(\C)$ by
$Rat(\C)$ in the proof of previous theorem, we obtain the following 

\begin{theorem} Let $R_1$ and $R_2$ be two complex
rational maps, and consider an automorphism of semigroups
$\Phi:Rat_{R_1}(\bar{\C})\rightarrow Rat_{R_2}(\bar{\C})$. Then there is 
$f\in Bij(\C)$ and $B\in PSL(2,\C)$ such that $\Phi(R)=f\circ R
\circ f^{-1}\circ B^{-1}$. In particular, if  $\Phi$ is continuous with
$\Phi(Id)=Id$, then $\Phi$ is conjugation by an element of $PSL(2,\C)$. 
\end{theorem}

Which implies the following

\begin{corollary} Two rational maps $R_1$ and $R_2$ are 
conjugate by a map in $PSL(2,\C)$ if, and only if, the semigroups
$Rat_{R_1}(\C)$ and $Rat_{R_2}(\C)$ are continuously isomorphic with an
isomorphism $\phi$, such that  $\phi(Id)=Id$.
\end{corollary}

By Theorem \ref{Thm.Sandwich}, the condition $\phi(Id)=Id$ is equivalent to
require that $\phi(R_1)=R_2$. Every automorphism of $Rat(\C)$ induces an
isomorphism of sandwich semigroups. Indeed, if $\phi\in Aut(Rat(\C))$ take 
$Q, R$ rational maps such that $\phi(Q)=R$, then $\phi$ is an isomorphism
between $Rat_Q(\C)$ and $Rat_R(\C)$. Let $\psi$ be  an isomorphism of sandwich
semigroups in
$Rat(\C)$. By Theorem \ref{Thm.Sandwich} and
Lemma~\ref{StWier.lemma}, $\psi$  induces an automorphism of $Rat(\C)$ if, and
only if, $\psi(Id)=Id$.
Let us now discuss the situation of sandwich isomorphisms for small semigroups.
Let $Q$ y $R$ be two non-constant rational maps, and consider the semigroup
$S=\langle Q,R, \C \rangle$. Take $R_1$ and $R_2$ in $S$ and consider an
isomorphism $\phi$ between $S_{R_1}$ and $S_{R_2}$. Since $Q$ and $R$ are the
non constant elements in $S$ with smaller degree, then we have either
$$\phi(Q)=R \textnormal{ and }\phi(R)=Q$$ or $$\phi(Q)=Q  \textnormal{ and }
\phi(R)=R.$$ In any case, $\phi^2$ fixes $Q$ and $R$. Then the
restriction of $\phi$ to constants  is a non trivial bijection of $\C$, which
commutes with $Q$ and $R$.

\section{Semigroup representations.}

In this section, we give examples of how the theory of  semigroup
representations applies to holomorphic dynamics. 
For every $X$, 
let us consider the decomposition of $Map(X)$ into the ideal of
constants, $\mc{I}(X)$, the group of bijections $Bij(X)$ and the rest $H(X)$.
That is $Map(X)=\mc{I}(X)\cup Bij(X)\cup H(X)$, as a consequence of Corollary 
\ref{Sch.coro}, it follows that every homomorphism of $Map(X)$ into $\mc{I}$
is constant. Similarly, the only homomorphism from $Map(X)$ to $Bij(X)$ is
the constant map with value $Id$.  

In the spirit of Lemma \ref{StWier.lemma}, we consider semigroups together
with the ideal of constants. Let $A$ be any set in $X$ and $S$ a subset of
$Map(X)$, then we denote by $\langle S,A\rangle$ the semigroup generated by $S$
and the constants in $A$ regarded as semigroups of $Map(\mc{O}^+_S(A))$,
where $\mc{O^+}_S(A)$ denotes the forward $S$-orbit of $A$.  With this
construction, the ideal of constants of $\langle S, A \rangle$ is precisely
$\mc{O}^+_S(A)\cup A$.

\begin{Example}
Let $f_0=z^2$, then $\langle f_0,1\rangle =1$, since $f_0=Id=1$ in $Map(\{1\})$.
Analogously, if $a$ is a periodic orbit of $f_0$, then $\langle f_0,a\rangle$
consists of the orbit of $a$ and the cyclic permutations of this orbit. 
\end{Example}

We can generalize the previous example to rational functions $R:\C 
\rightarrow \C$. In this case, we obtain a family of semigroups $\langle
R,a\rangle$ parametrized by a point $a$ in the plane $\C$. In this way, the set
$\mathcal{D}_R=\{\langle R,a\rangle:a\in \C\}$ inherits the usual topology from
$\C.$ Let $\mc{X}_R\subset \mathcal{D}_R$ be the set of finite semigroups,
we call the set  $\mc{J}_R=\overline{\mc{X}_R}\setminus\{\textnormal{isolated
points}\}$, the \textit{algebraic Julia set} of $R$. The complement
$\mc{F}_R=\{\langle R,a\rangle:a\in \C\}\setminus \mc{J}_R$ will be called the
\textit{algebraic Fatou set} of $R$ in $\mathcal{D}_R$. In this setting, the
algebraic Fatou set is the interior of the set of free semigroups in
$\mathcal{D}_R$. These definitions reflect the dynamical Julia set $J(R)$,
which is the closure of the repelling periodic points in $\C$ and, the dynamical
Fatou set $F(R)$ which is the complement of $J(R)$ in $\C$.

\subsection{Representations of semigroups of polynomials.}

Let $\mc{P}$ be a partition of $Pol(\C)$, we say that
$\mc{P}$ is a \textit{compatible partition} if for $A,B\in \mc{P}$, and a
pair of points $a\in A$, $b\in B$, the composition $a\circ b$ belongs to a
component $C$ in $\mc{P}$ which do not depend on the representatives $a$ and
$b$. A \textit{graduation} is a partition of $Pol(\C)$ which is compatible
with composition.

As we discussed earlier in the paper, $Pol(\C)$ has a non empty set of
multiplicative characters.  Each multiplicative character in $Pol(\C)$ induces a
graduation in $Pol(\C)$. 
The fibers of multiplicative characters induce compatible partitions. In
particular, the degree of a polynomial induces a compatible partition of
$Pol(\C)$. In this case,  the classes of this partition are $Pol_d(\C)$, the
set of polynomials of given degree $d$. We will describe now some examples
of representations of semigroups of the form $\langle P, A\rangle$ into
$Pol_d(\C)$.

Note that since we are including an ideal of constants $A$ in the
domain, then we have to include the constants in $Pol_d(\C)$ as well.
Otherwise, there is no representation from $\langle P, A\rangle$ into
$Pol_d(\C)$. Nevertheless, including constants, in both domain and range, is
consistent with the philosophy of Lemma \ref{StWier.lemma}. In this setting, 
every representation of $S$ in $Pol(\C)$ is geometric, and realized by a map
defined in the complex plane.
Let $S$ be a subsemigroup in $Pol(\C)$ containing $Id$, and let us consider
representations of $S$ into $Pol_0(\C)$, the semigroup of constant polynomials.
Let $\phi:S\rightarrow Pol_0(\C)$ be a homomorphism, since $\phi(Id)$ is
constant, then $\phi(R)=\phi(Id)\circ \phi(R)=\phi(Id)$. Hence, any
representation of $S$ into the constant polynomials is a constant map. 

The theory of representation of semigroups of the form $\langle P, J(P)\rangle$
into $Pol(\C)$ is widely discussed in holomorphic dynamics in other terms. For
example, the theory of the continuous representations of  $\langle P,
J(P)\rangle$ into $Pol(\C)$ is parameterized by the $J$-stable components of
$P$. For example, see \cite{Cabre.Maki.Teic}.

Another important situation is representations of semigroups $\langle P,
\mathcal{P}(P)\rangle$ into $Pol(\C)$, here $\mathcal{P}(P)$ is the
postcritical set of $P$. Interior components of the representation
space can be parameterized by combinatorially equivalent polynomials.
Uniformization of these components by suitable geometric objects (like
suitable Teichm\"{u}ller spaces), shed light on many problems in holomorphic
dynamics. In this direction, important advances were made by Douady, Hubbard,
Lyubich, McMullen, Sullivan and Thurston, among many others. See for example
\cite{DHTop} and \cite{Mc1}.

Now, let us consider the space of representations of affine semigroups into the
space of polynomials of degree $d$. This space includes all linearizations  
around periodic orbits. Here, we review the repelling case. A complete treatment of 
linearization theory in holomorphic dynamics can be found in  Milnor's book
\cite{Mdyn}.

Let $A_\lambda$ in $ \textnormal{Aff}(\C)$ of the form $z\mapsto \lambda z$. 
Let
$P$ be a polynomial such that there exist a repelling cycle
$\mc{O}=\{z_0,z_1,...,z_n\}$ with multiplier $\lambda$. The Poincar\'e function
associated to $z_0$, is a map $\phi:\C\rightarrow \C$  sending $0$ to $z_0$
which locally conjugates 
$A_\lambda$ to $P^n$ around $z_0$. This construction induces a representation of
$ \langle A_\lambda, \C \rangle$ into $ \langle P,U_0\rangle$ for a suitable
neighborhood $U_0$ of $z_0$. Moreover, since Poincar\'e functions turn out  to be
meromorphic functions, it also induces a representation of affine
semigroups into the semigroups of meromorphic functions. Similar constructions
apply to other kind of linearizations. In the attracting case, the inverse of
the Poincar\'e function, defined on a neighborhood $U_0$ of $z_0$, is known as
K\"{o}nig's coordinate and gives a representation of $ \langle P,U_0\rangle$
into $\langle A_\lambda, \mathbb{D}_r\rangle$, where $\mathbb{D}_r$ denotes the
disk of radius $r$ and $r<1$. This construction can also be applied to the
parabolic case.

The process of renormalization, in holomorphic dynamics, gives examples of
semigroups of the form $ \langle P,U\rangle$ that admit
representations into themselves.

Let $P$ be a polynomial $P(z)$, of degree $n$, with connected and locally
connected Julia set.  Then, $\infty$ is a superattracting fixed point of $P$.
If $A_0(\infty)$ denotes the basin of $\infty$ of $P$, by B\"{o}ttcher's 
theorem,
there is a homeomorphism $\phi:\bar{\C}\setminus \bar{\mathbb{D}}\rightarrow
A_0(\infty)$, that conjugates $z\mapsto z^n$ in $\bar{\C}\setminus
\bar{\mathbb{D}}$ with $P$ in $A_0(\infty)$. Since $J(P)$ is connected and
locally connected, the map $\phi$ extends to the boundaries by Caratheodory's 
theorem. The map on the boundaries induces a representation of $\langle z^n,
\mathbb{S}^1 \rangle$
into $\langle P, J(P) \rangle$.

It would be interesting to have results, analogous to Theorem \ref{Gal.thm} or
Theorem \ref{Thm.Sandwich},
that characterizes the action of quasiconformal maps in $\C$. This would allow
us to determine quasiconformal conjugation in terms of semigroup
representations.

\subsection{When Julia set is homeomorphic to a Cantor set.}

Now consider the special case where $J(R)$ is homeomorphic to a Cantor set. For
simplicity in the arguments, let us assume that $deg(R)=2$. 

Let us consider a Jordan curve $\gamma$ containing in its interior the Julia
set and a critical point; while the other critical point and all critical values
lie outside $\gamma$. Choosing a suitable $\gamma$, we assume that
$R^{-1}(\gamma)$ is
contained in the interior of $\gamma$ and consists of two Jordan curves
$\gamma_1$ and $\gamma_2$. We get an scheme similar to the one sketched in
Figure \ref{figure.cantor}. Let us call $D_1$ and $D_2$ the interiors of
$\gamma_1$ and $\gamma_2$, respectively. 

\begin{figure}[htbp]
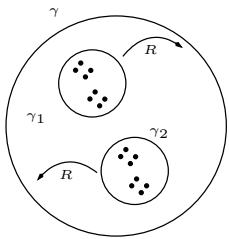

\begin{center}
\include{cantor.pspdftex}
\caption{Cantor scheme.}
\label{figure.cantor}
\end{center}
\end{figure}

With this scheme, we obtain representations of $\langle R,D\rangle$ into
other semigroups. To do so, let us modify topologically the Cantor scheme, and 
instead of the restrictions of $R$ on $D_i$, consider affine
maps  $A_i$ sending the modified $\gamma_i$ to $\gamma$. This induces a
representation of $\langle R,J(R)\rangle$ into $\textnormal{Aff}(\C)$. 

If we modify the curves $\gamma$, $\gamma_1$ and $\gamma_2$ to circles and 
considering M\"{o}bius transformations $g_i$, instead of the maps that send
$\gamma_i$ to $\gamma$. We get a representation $\Phi$ of $\langle
R,J(R)\rangle$
into a ``half'' classical Schottky group $\Gamma$ with two generators. This is
an example of a representation of non cyclic Kleinian groups in 
rational semigroups. The
conjugating map of $\Phi$ may be taken quasiconformal, hence 
the Hausdorff dimension of the limit set of the Schottky group can be estimated in terms of 
the Hausdorff dimension of $J(R)$.  In particular,
let $R(z)$ be a quadratic polynomial of the form $z^2+c$, such that $J(R)$ is a
Cantor set. In this case, the parameter $c$ belongs to the complement of the
Mandelbrot set. A theorem of Shishikura shows that there are sequences of
quadratic polynomials $R_{c_n}(z)=z^2+c_n$, with parameters $c_n$ tending to
the boundary of the Mandelbrot set, and such that the Hausdorff dimension of
the Julia sets tends to 2. With this result, Shishikura showed that the
Hausdorff dimension of the boundary of the Mandelbrot set is 2. Perhaps, using
the representation above is possible to get a result analogous of Shishikura's
theorem for the boundary of the Classical Schottky space.  

It is interesting to solve the extremal problem between these two objects from
holomorphic dynamics. In case there exist an extremal map from $\langle
R,J(R)\rangle$ into the Classical Schottky space, there would be  a sort of
estimate from above of the distance between this two pieces of
Sullivan's dictionary.

The problem to describe the set of representations of $\langle P,A\rangle$, for
an invariant set $A$, into $\langle \textnormal{Aff}(\C),\C\rangle$ is
difficult, still
remain many  questions.
In the case where $S\in Rat(\C)$, it is interesting to understand the space
of representations of $\langle S,A\rangle$ into $PSL(2,\C)$. 

\subsection{Binding semigroups of maps with constants.}

Let us consider two semigroups of the form $S_1=\langle g_1, A_1\rangle$
and $S_2=\langle g_2, A_2\rangle$, in this case the categorical sum, or
coproduct,  
$S_1\coprod S_2$, is defined as $ \langle g_1\coprod g_2, A_1\times \{1\} \sqcup
A_2\times\{2\}\rangle$, where $g_1\coprod g_2$ is a map defined on the disjoint
union $A_1\times \{1\}\sqcup A_2\times \{2\}$  by $$g_1\coprod g_2(x)=\left\{
\begin{array}{ll} g_1(x) &  \textnormal{ if } x\in A_1 \\g_2(x) &
\textnormal{ if } x\in A_2 \\ \end{array} \right. $$
Analogously, we define the binding of a countable family of semigroups of the 
form $\langle g_i, A_i\rangle$. A classical example of such construction in
dynamics is the process of mating of quadratic polynomials, first described by
 Douady in \cite{Douady}.  We start with two mateable polynomials
$S_1=\langle P_1, \mathbb{C}\rangle $ and $S_2=\langle P_1, \mathbb{C}\rangle $.
Using a topological construction, the mating $P_1\coprod P_2$ is a quadratic
rational map $\langle R, \C\rangle$. Thus we have a representation of 
$S_1\coprod S_2$ into the space of rational maps of degree $2$.

\subsubsection{Simultaneous linearizations and deformation spaces.}
Let us now discuss a more elaborated example, associated to a fixed rational map
$R_0$ of degree $d$. Let $\{a_0, a_1,..., a_{n-1}\}$ be a periodic cycle of
$R_0$, of period $n$ and multiplier $\lambda$, with $\lvert\lambda\lvert>1$. Let
us denote by $A_\lambda$ the map
$z\mapsto \lambda z$, and $\phi$ the Poincar\'{e} function associated to $R_0^n$
and $a_0$. As we discussed above $\phi$ induces a representation of $\langle
A_\lambda, \C\rangle $ into $\langle R_0, \C\rangle $. The same is true for
$R^i_0\circ \phi$, for each $i=0,...,n-1$,  all together, induce a
representation of the binding $\langle \coprod_{i=0}^{n-1}
A_{\lambda}, \sqcup_{i=0}^{n-1} \C\times \{i\}\rangle$ into $\langle R_0,
\C\rangle $, here
we put a component  $\langle A_\lambda, \C\rangle $ for each periodic point in
the cycle $\{a_0, a_1,..., a_{n-1}\}$. Let us carry this construction further 
considering all repelling periodic cycles of $R_0$,  we obtain a countable
binding of semigroups of the form  $\langle A_\lambda, \C\rangle $ associated
to all Poincar\'e functions of $R_0$. Let us call
$\mathcal{A}(R_0)$ this countable binding, so we have a representation 
$\Psi:\mathcal{A}(R_0)\rightarrow \langle R_0, \C\rangle $. Taking instead
of $\langle R_0, \C \rangle$,  the corresponding Poincar\'{e}
functions, we obtain a representation $\tilde{\Psi}$ from $\mathcal{A}(R_0)$ into
$Mer(\C)$.  The image $\phi(\mc{A}(R_0))$ has a
compactification which is related to Lyubich-Minsky laminations discussed in
Section 7 of \cite{LM}.

Let us assume that $R_0$ is hyperbolic of degree  $d$. Since $\langle R_0, \C \rangle$ 
is a subsemigroup of $Rat(\C)$, let us now regard $\Psi$ as a homomorphism from
$\mc{A}(R_0)$ into $Rat(\C)$. Let $\mc{X}(R_0)$ be the space  of representations
from $\mc{A}(R_0)$ into $Rat(\C)$, whose image is of the form $\langle R, \C
\rangle$ for some $R$ of degree $d$. In other, words we are considering all
graduated representations that arise by
deformations of the semigroup $\langle R_0,\C\rangle$. Let us define the map
$P:\mc{X}(R_0)\rightarrow Rat_d(\C)$, such that for every $\Phi\in \mc{X}(R_0)$,
let $P(\Phi)=R$ where $R$ is the non constant rational map generating
$P(\Phi)$.

Let $Par_d(\C)$ be the set of all rational maps, of degree $d$, that admit a
parabolic periodic point. Then $P(\mc{X}(R_0))$, in $Rat_d(\C)$,  is equal to
$Rat_d(\C) \setminus
\overline{Par_d(\C)}$. By a result of Lyubich, see
\cite{L}, the space $Rat_d(\C) \setminus \overline{Par_d(\C)}$ consists of the
union of $J$-stable components in $Rat_d(\C)$. 

In  \cite{Cabre.Maki.Teic}, the authors construct a  dynamical Teichm\"{u}ller
space $T_2(R_0)$, which uniformize the $J$-stable components of $R_0$. It turns
out that the space $T_2(R_0)$  is isomorphic to  $\mathcal{X}(R_0)$. 
 
\section{Correspondences.}

Let $A$ and $B$ be two sets, let $G$ be a subset of $A\times B$. A
\textit{correspondence} is a triple $(G,A,B)$. If $(a,b)\in G$ we say that $b$
corresponds to $a$ under $G$.  The notion of correspondences
generalizes, in a way,  the notion of functions. Indeed, for every map $f:X\rightarrow Y$,
the graph of $f$ induces a correspondence in $X\times Y$. Borrowing notation
from Function Theory, we define the set $$Im(G)=\{b\in B:\exists a\in A
\textnormal{ such that } (a,b)\in G\}$$ is called the \textit{image} of $G$, analogously
the \textit{domain} of $G$ is defined by $$Dom(G)=\{a \in A:\exists b \in B
\textnormal{ such that } (a,b)\in G\}.$$ For every
$b\in Im(G)$ we call $G^{-1}(b)=\{a\in A: (a,b)\in G\}$ the preimage of $b$
under $G$. Similarly, the image of an element $a\in A$ is the set
$G(a)=\{b\in B: (a,b)\in G\}.$  
Given a set $G\subset A\times B$, the set $G^{-1}=\{(b,a)\in B\times A: (a,b)\in
G\}$ is called the inverse of $G$. Let $G_1\subset A\times B$ and $G_2\subset
B\times C$ be two correspondences, the composition $G_1\circ G_2$ of $G_1$ and
$G_2$ is the correspondence induced by the set
$$G_2\circ G_1=\{(a,c)\in A\times C: \exists b\in B \textnormal{ such that }
(a,b)\in G_1 \textnormal{ and } (b,c)\in G_2 \}.$$ 
Let $X$ be a set, a correspondence $K$ in $X$ is a correspondence of the form
$(K,X,X)$, additionally we require that  $Dom(K)=X$.  A correspondence $K$ in
$X$ is called \textit{surjective} if $Im(K)=X$ and, \textit{finite} if every
image is a finite set. In particular, constant maps are finite correspondences.

If $G$ is a finite correspondence, the degree of the image of $G$
is the maximum of the cardinalities of its images. 

\subsection{Schreier's Lemma for correspondences.}

With  composition, the set of correspondences $Corr(X)$ in
$X$ is a semigroup. Since functions are special cases of
correspondences, the semigroup of correspondences 
of $X$ contains $Map(X)$. The proof of the following lemma 
is immediate by contradiction.

\begin{lemma}\label{maps.corr} Let $K_1$ and $K_2$ be two correspondences in
$Corr(X)$ such that $g=K_1\circ K_2$ is a map and $K_2$ is surjective, then
$K_1$ is a map. \end{lemma}

We will start by generalizing Schreier's  lemma restricted to
correspondences generated by maps. Let us start with some definitions,

\begin{definition} A correspondence $K$ in a set $X$ is called a
\textnormal{block} if $K$ has the form $R_1\circ R_2^{-1}$,
where $R_1$ and $R_2$ belong to $Map(X)$ and $R_2$ is surjective.
\end{definition}

We denote by $\mathcal{BL}(X)$ the subsemigroup of $Corr(X)$, 
generated by all block correspondences.  

\begin{theorem}[Schreier Lemma for blocks]\label{Sch.L.Bl} Let
$\phi:\mc{BL}(X)\rightarrow \mc{BL}(Y)$ be an homomorphism then, there exist
$f\in Map(X,Y)$ such that for every $K\in \mc{BL}(X)$ we have
$\phi(K)=f\circ K\circ f^{-1}$.
\end{theorem}

\begin{proof} The identity $Id$ is characterized among $Corr(X)$ by the 
properties $Id\circ Id=Id$ and that for every $C\in Corr(X)$ we have 
$Id\circ C = C\circ Id=C$. Since this properties are preserved by homomorphisms
we have $\Phi(Id)=Id$. Let $R\in Corr(X)$ be any map, then $R\circ R^{- 1}=Id$.
But then $\Phi(R\circ R^{- 1})=\Phi(R)\circ \phi(R^{-1})=Id$ is a map, by
Lemma \ref{maps.corr} then $\Phi(R)$ is a map. Hence $\Phi$ sends maps into
maps, so $\phi$ restricted to $Map(X)$ is a homomorphism of semigroups. By
Lemma \ref{StWier.lemma}, there exist $f\in Map(X,Y)$ such that, for every
map $R$, $\Phi(R) \circ f=f\circ R$. Then $\Phi(R)=f\circ R\circ f^{-1}$ for
all maps $R$.

Since blocks generate $\mc{BL}(X)$, it is enough to check that the theorem holds
for every correspondence of the form $K=R_1^{-1}$, where
$R_1$ is a map. Since $R_1\circ R_1^{-1}=Id$ we have $$\Phi(R_1\circ
R_1^{-1})=\Phi(R_1)\circ \Phi(R_1^{-1})=Id$$ on the other hand,
$$\Phi(R_1)=f\circ R_1^{-1}$$ then $$f\circ R_1 \circ
f^{-1}\circ\Phi(R_1^{-1})=Id$$ it follows that $$\Phi(R_1^{-1})=f\circ
R_1 ^{-1}\circ f^{-1}$$ and then for every block $K$, $\Phi(K)=f\circ K\circ
f^{-1} $ as we wanted to prove.
\end{proof}

We now include in the discussion the constant maps in $Corr(X)$, these are no
longer an ideal, but we can consider the unique \textit{minimal left ideal} 
$\mathcal{I}$ in $Corr(X)$, which is  generated by all constant
maps. The semigroup of correspondences acts on $\mathcal{I}$. That
is, there is a map $\alpha:Corr(X)\rightarrow Map(\mc{I})$ that sends every
correspondence $K\in Corr(X)$ to the left translation by $K$ in $Map(\mc{I})$.

\begin{lemma}\label{alpha.lemma}
 The map $\alpha:Corr(X)\rightarrow Map(\mc{I})$ is a one-to-one map. Moreover, 
for every $c\in \mc{I}$, we have $\alpha(c)=c$.
\end{lemma} 
\begin{proof}
Suppose that $K_1$ and $K_2$ are correspondences in $Corr(X)$ such that
$\alpha(K_1)=\alpha(K_2)$. In particular, for every constant $c\in \mc{I}$, we
have $K_1\circ c=K_2 \circ c$. However, a correspondence is characterized by the set
of images,  then $K_1=K_2$. The second part of the Lemma follows from the
equation $c\circ K= c$ for all $c\in \mc{I}$.  
\end{proof}

Now we are set to prove:

\begin{theorem}\label{Thm.Sch.Corr}[Schreier's Lemma for correspondences] Let
 $$\Phi:Corr(X)\rightarrow Corr(Y)$$ be a homomorphism of semigroups. Then,
there is a map $f\in Map(X,Y)$, such that, for every $K\in Corr(X)$ we
have
$\Phi(K)=f\circ K \circ f^{-1}$. 
 \end{theorem}

\begin{proof}
 By the same argument in the proof of Theorem \ref{Sch.L.Bl}, the map $\Phi$
sends maps to maps. Moreover, the restriction of $\Phi$ to $\mathcal{BL}(X)$, is
an homomorphism from $\mathcal{BL}(X)$ to $\mathcal{BL}(Y)$, by Theorem
\ref{Sch.L.Bl} there is $f\in Map(X,Y)$ such that 
for every $K\in \mc{BL}(X)$, we have that $\phi(K)=f\circ K\circ f^{-1}$.

Let $\mc{I}$ and $\mc{J}$ denote the minimal ideals in $Corr(X)$ and
$Corr(Y)$, respectively. Let us consider the maps $\alpha_X:Corr(X)\rightarrow
Map(\mc{I})$, and $ \alpha_Y:Corr(Y)\rightarrow
Map(\mc{J})$ as in Lemma \ref{alpha.lemma}, and define $S_X=\alpha(Corr(X))$
and $S_Y=\alpha(Corr(Y)).$ By Lemma \ref{alpha.lemma}, the maps $\alpha_X$ and
$\alpha_Y$ are bijections to their images.
Moreover, $\alpha_X$ and $\alpha_Y$ send constants to constants. Hence the map 
$\alpha_Y\circ \phi \circ\alpha_X^{-1}$ is a homomorphism between the semigroups
$G_X$ and $G_Y$, sending constants to constants. By Corollary \ref{Sch.coro},
there exist 
$F\in Map(\mc{I},\mc{J})$ such that for every $g\in G$, we have $$\alpha_Y\circ
\phi  \circ\alpha_X^{-1}(g)=F\circ g\circ F^{-1} .$$
Since $\alpha$, restricted to minimal ideal is the identity, then for
every $c$ in $\mc{I}$ and every correspondence $K\in Corr(X)$
we have $F(c)=f(c)$,  also
$$\alpha_Y(K)\circ c=K\circ c,$$ and  $$ \alpha_X^{-1}(c)=c$$
evaluating in $c$ the equation above, we get 
$$(\alpha_Y\circ \phi(K)\circ \alpha_X^{-1})\circ c=
\phi(K)\circ c$$$$= F\circ K \circ F^{-1}\circ c.$$ But then
$\phi(K)=f\circ K\circ f^{-1}$.

\end{proof}

Note that in the proof of Theorem \ref{Thm.Sch.Corr}, we need the theorem on
block correspondences to get the existence of the map $f$.  Once we have
Schreier's lemma for the whole semigroup of correspondences, we can generalize
it for subsemigroups of correspondences, as long as they contain the minimal
ideal of constants.

\begin{corollary}\label{Schreier.Corr.corol}
 Let $S_1$ and $S_2$ be subsemigroups of  $Corr(X)$ and $Corr(Y)$,
respectively, such that $X_1=S_1\cap X$ and $Y_1=S_2\cap Y$ are both non
empty. If $\phi:S_1\rightarrow S_2$ is a homomorphism of semigroups, then there
is $f:X_1\rightarrow Y_1$, such that for all $K\in S_1$, $\phi(K)=f\circ
K\circ f^{-1}$. Moreover,
\begin{itemize}
 \item the homomorphism $\phi$ is injective, or surjective, if and only
if, the map $f$ is injective or surjective. In particular, $\phi$ is an
isomorphism if, and only if, $f$ is a bijection.

\item When $S_1$ and $S_2$ are
topological semigroups, then $\Phi$ is continuous if, and only if, $f$ is
continuous.

\end{itemize}

\end{corollary}

\subsection{Holomorphic correspondences.}

A correspondence $K$ in $\C$ is \textit{holomorphic} if, as a set of $\C\times
\C$, 
$K$ can be decomposed as a countable union of analytic varieties, see
McMullen's book \cite{Mc2}. However recall that, in our setting, we
require that $Dom(K)=\C$. Moreover, we assume  that the preimage of
every point admits an analytic extension to the whole Riemann sphere, with the
exception
of finitely many points. Let us denote by $HCorr(\C)$ the semigroup of holomorphic
correspondences, which includes the semigroup of entire maps and constants. 
We denote by $FHCorr(\bar{\C})$, the semigroup of finite correspondences on the Riemann sphere.
By definition $FHCorr(\bar{\C})$ contains the  semigroup of rational maps $Rat(\C)$ together
with all constant maps. Hence, there exist a minimal left ideal of finite
holomorphic correspondences. Since rational maps are onto the Riemann sphere,
if $R_1$ and $R_2$ are rational maps, the block $R^{-1}_1\circ R_2$  belongs to $FHCorr(\C)$.

Let $K\in FHCorr(\bar{\C})$, a holomorphic correspondence with degree $d$. That 
is there is $z$ such that $K(z)$ consists of $d$ points. Let $S_1,S_2, ...,
S_{d}$ denote all the symmetric polynomials with $d$ variables. For every $i$,
$S_i(K)$ induces a holomorphic map from $\bar{\C}$ to $\bar{\C}$, it follows
that $S_i(K)$ is a rational map.  Moreover, for every $z$ the image $K(z)$
are the roots of the polynomial $$S_1(K(z))+S_2(K(z))Z+...+S_d(K(z)) Z^{d-1}
+Z^d.$$
Reciprocally any polynomial in $Z$, whose coefficients are rational maps in $z$, defines a
finite holomorphic correspondence in $\bar{\C}$. From this discussion,  we have
the following known fact.

\begin{proposition}\label{finite.corr.pol}
The space $FHCorr(\bar{\C})$ is equivalent to the space of monic polynomials
with coefficients in $Rat(\C)$.
\end{proposition}
 
The proof of Theorem \ref{Thm.Sch.Corr}, can be repeated in the setting of
holomorphic correspondences. In this case, every automorphism of $HCorr(\C)$ or
$FHCorr(\bar{\C})$ is induced by conjugation of some function in $Bij(\C)$ or
$Bij(\bar{\C})$. Nevertheless, the holomorphic structure imposes holomorphic
conditions in such bijections.

\begin{theorem} The following statements are true
\begin{itemize}
 \item Every automorphism of $HCorr(\C)$ is continuous. Moreover,
$$Aut(HCorr(\C))\simeq \textnormal{Aff}(\C).$$ 
\item The action of $Gal(\C)$ extends to an action in $FHCorr(\bar{\C})$. In fact, $$Aut(FHCorr(\bar{\C}))\simeq \langle PSL(2,\C), Gal(\C)\rangle.$$
\end{itemize}

\end{theorem}

\begin{proof}

The semigroup of maps in $HCorr(\C)$ coincides with the semigroup of entire
maps. The first part of the theorem is a consequence of Corollary
\ref{cor.mer.cont}.
Since the semigroup of maps in $FHCorr(\bar{\C})$ is equal to $Rat(\bar{\C})$.
By restriction, any automorphism of $FHCorr(\bar{\C})$ induces an automorphism
of $Rat(\bar{\C})$. But every automorphism of $Rat(\bar{\C})$ is 
generated by $PSL(2,\C)$ and $Gal(\C)$. Now let us see that, in fact, $Gal(\C)$
also acts on $FHCorr(\bar{\C})$. 
Let $\gamma$ an element in $Gal(\C)$, and let $K\in FHCorr(\bar{\C})$, then 
$\gamma \circ K \circ \gamma^{-1}$ is a finite correspondence in $Corr(\bar{\C})$.
Let $d$ be the maximum cardinality of a fiber of $K$. Remind that $K$ is holomorphic
in the Riemann sphere if, and only if,  there is a symmetric polynomial  $S_d$
in $d$ variables, such that $S_d(K)$ is a rational map in $\bar{\C}$. Since
$\gamma$ acts on symmetric polynomials, there is a symmetric polynomial
$\tilde{S}_d$ such that $$\tilde{S}_d(K)=\gamma \circ S_d(K) \circ
\gamma^{-1}=S_d(\gamma \circ K \circ \gamma^{-1})$$ But the second equality is
the conjugation of a rational map by a Galois map, hence is rational. This
implies that $\gamma \circ K \circ \gamma^{-1}$ is a holomorphic correspondence.
It follows that the group of automorphisms of $FHCorr(\bar{C})$ is isomorphic to the 
group of automorphisms of $Rat(\C)$, which by Proposition \ref{prop.aut.rat} is isomorphic to
$\langle PSL(2,\C), Gal(\C)\rangle$.
 \end{proof}

The central argument for Theorem \ref{Thm.Sandwich} is Schreier's Lemma, with
some modifications we can prove the corresponding theorem for holomorphic
correspondences. 

\begin{theorem}\label{Thm.Sandwich.corr} Let $K_1$ and $K_2$ be two holomorphic
correspondences. Let $$\Phi:Corr_{K_1}(\C)\rightarrow Corr_{K_2}(\C)$$ be an
isomorphism of sandwich semigroups. Then there is $f\in Bij(\C)$, and $B\in
\textnormal{Aff}(\C)$, such that $\Phi(P)=f\circ P \circ  f^{-1}\circ B^{-1}$. 

\end{theorem}

It is not clear whether the Galois group action acts on holomorphic
correspondences. Perhaps there is a generalization to Hinkkanen's argument
in this setting. Now, we can state an  analogous statement to  Theorem
\ref{Gal.thm} for $FHCorr(\bar{\C})$.

\begin{corollary} Let $F$ be an element in $Bij(\bar{\C})$, that fixes $0,1$
and $\infty$. Then $F$ belongs to $Gal(\C)$ if, and only if, $F$  induces an
automorphism of
$FHCorr(\bar{\C})$.
\end{corollary}

\bibliographystyle{amsplain}
\bibliography{workbib}

\providecommand{\bysame}{\leavevmode\hbox to3em{\hrulefill}\thinspace}
\providecommand{\MR}{\relax\ifhmode\unskip\space\fi MR }
\providecommand{\MRhref}[2]{%
  \href{http://www.ams.org/mathscinet-getitem?mr=#1}{#2}
}
\providecommand{\href}[2]{#2}
\begin{thebibliography}{10}

\bibitem{BeaRitt}
A.~F. Beardon and T.~W. Ng, \emph{On {R}itt's factorization of polynomials}, J.
  London Math. Soc. (2) \textbf{62} (2000), no.~1, 127--138.

\bibitem{Cabre.Maki.Teic}
C.~Cabrera and P.~Makienko, \emph{On dynamical {T}eichmuller spaces},
  arXiv:0911.5715 (2009).

\bibitem{Douady}
A.~Douady, \emph{Syst{\`e}mes dynamiques holomorphes}, Bourbaki seminar, {V}ol.
  1982/83, Ast{\'e}risque, vol. 105, Soc. Math. France, Paris, 1983,
  pp.~39--63.

\bibitem{DHTop}
A.~Douady and J.~H. Hubbard, \emph{A proof of {T}hurston's topological
  characterization of rational functions}, Acta Math. \textbf{171} (1993),
  263--297.

\bibitem{EreEnd}
A.~Eremenko, \emph{On the characterization of a {R}iemann surface by its
  semigroup of endomorphisms}, Trans. Amer. Math. Soc. \textbf{338} (1993),
  no.~1, 123--131.

\bibitem{Hinkentire}
A.~Hinkkanen, \emph{Functions conjugating entire functions to entire functions
  and semigroups of analytic endomorphisms}, Complex Variables and Elliptic
  Equations \textbf{18} (1992), no.~3-4, 149--154.

\bibitem{L}
M.~Lyubich, \emph{Dynamics of the rational transforms; the topological
  picture}, Russian Math. Surveys (1986).

\bibitem{LM}
M.~Lyubich and Y.~Minsky, \emph{Laminations in holomorphic dynamics}, J. Diff.
  Geom. \textbf{47} (1997), 17--94.

\bibitem{Mag1}
K.~D. Magill, \emph{A survey of semigroups of continous self maps}, Semigroup
  Forum \textbf{11} (1975/76), 189--282.

\bibitem{Mc1}
C.~McMullen, \emph{Complex dynamics and renormalization}, Annals of Mathematics
  Studies, vol. 135, Princeton University Press, Princeton, NJ, 1994.

\bibitem{Mc2}
\bysame, \emph{Renormalization and 3-manifolds which fiber over the circle},
  Annals of Mathematics Studies, vol. 142, Princeton University Press,
  Princeton, NJ, 1996.

\bibitem{Mdyn}
J.~Milnor, \emph{Dynamics of one complex variable}, Friedr. Vieweg \& Sohn,
  1999.

\bibitem{Ritt}
J.~F. Ritt, \emph{Prime and composite polynomials}, Trans. Amer. Math. Soc.
  \textbf{23} (1922), no.~1, 51--66.

\bibitem{Schre.end}
J.~Schreier, \emph{Uber {A}bbildungen einer abstrakten {M}enge auf ihre
  {T}eilmengen}, Fund. Math. (1937), no.~28, 261--264.

\end{thebibliography}

\end{document}